\documentclass[reqno,12pt]{amsart}
\usepackage{amsfonts}
\usepackage{graphicx}
\usepackage{amsfonts,amsmath, amssymb}
\usepackage{marginnote}
\usepackage{color}
\usepackage{cite}
\usepackage{bm}
\usepackage{enumerate}
\numberwithin{equation}{section}

\theoremstyle{plain}
\newtheorem{theorem}{Theorem}[section]

\newtheorem{lemma}[theorem]{Lemma}

\newtheorem{definition}[theorem]{Definition}
\theoremstyle{remark}
\newtheorem{remark}[theorem]{Remark}

\begin{document}
	
	\date{\today}
	\title[Fields of cancellation]{Vector fields of  Cancellation  for the Prandtl Operators}
	
	\author[T. Yang]{Tong Yang}
	\address[T. Yang]{Department of Mathematics, City University of Hong Kong, Hong Kong, China}
	\email{matyang@cityu.edu.hk}

\begin{abstract}
	It has been a fascinating topic in the study of boundary layer theory about the well-posedness of Prandtl equation that was derived in  1904. Recently, new ideas  about
	cancellation to overcome the loss of tangential derivatives were  obtained so that Prandtl equation can
	be shown to be well-posed in Sobolev spaces to avoid the use of 
	Crocco transformation as in the classical work of Oleinik. This short note aims to show that the cancellation mechanism is in fact related to some intrinsic
	directional derivative that can be used to recover the tangential derivative under some structural assumption on the fluid
	near the boundary. 
\end{abstract}


\maketitle

\section{Introduction}

In 1904, Prandtl  derived the famous equation to describe the fluid behaviour near a boundary by
resolving the difference between the viscous and the inviscid effects
with no-slip boundary condition. This revolutionary result has vast applications in aerodynamics and other areas of engineering. It  also provides a typical mathematical model that attracts attention even now because a lot of mathematical problems remain unsolved. The key observation by Prandtl is that
outside a layer of thickness of $\sqrt{\frac{1}{\mbox{Re}}}$, convection dominates
so that the flow is governed by the
Euler equations; while inside a layer (boundary layer) of thickness of $\sqrt{\frac{1}{\mbox{Re}}}$, convection
and viscosity balance so that the flow is governed by the Prandtl  equations. Here $\mbox{Re}$ is the
Reynolds number.

Let us briefly recall the derivation of the Prandtl equation. Consider the 
incompressible Navier-Stokes equations over a flat boundary  $\{(x,y)\in D, z=0\}$ with
no-slip boundary condition,
$$
\begin{cases}
&\hspace{-.15in} \partial_t {\mathbf u}^\epsilon+({\mathbf u}^\epsilon\cdot
\nabla){\mathbf u}^\epsilon+\nabla p^\epsilon-\epsilon\mu\Delta {\mathbf u}^\epsilon=0, \\
&\hspace{-.15in} \nabla\cdot {\mathbf u}^\epsilon=0,\\
&\hspace{-.15in} 
{\mathbf u}^\epsilon|_{z=0}=0,
\end{cases}
$$
where ${\mathbf u}^\epsilon$ is the velocity field, $p^\epsilon$ represents the pressure
and $\epsilon\mu$ is the viscosity coefficient with $\epsilon$ being a small parameter.
According to the Prandtl ansatz, set
${\mathbf u}^\epsilon=(u^\epsilon, v^\epsilon, w^\epsilon)^T$ with
the following scaling:
$$
\begin{cases}
u^\epsilon(t,x,y,z)=u(t,x,y,\frac{z}{\sqrt{\epsilon}})+o(1),\\
v^\epsilon(t,x,y,z)=v(t,x,y,\frac{z}{\sqrt{\epsilon}})+o(1),\\
w^\epsilon(t,x,y,z)=\sqrt{\epsilon}w(t,x,y,\frac{z}{\sqrt{\epsilon}})+o(\sqrt{\epsilon}).
\end{cases}$$
The leading order gives the following classical Prandtl equations
$$
\begin{cases}
\partial_t u+(u\partial_x+v\partial_y+w\partial_z)u+\partial_xp^E(t,x,y,0)=\mu\partial_z^2u,\\
\partial_t v+(u\partial_x+v\partial_y+w\partial_z)v+\partial_yp^E(t,x,y,0)=\mu\partial_z^2v,\\
\partial_xu+\partial_yv+\partial_z w=0,\\
(u, v, w)|_{z=0}=0, \qquad \lim\limits_{z\to+\infty}(u, v)=(u^E, v^E)(t,x,y,0),
\end{cases}$$
where the fast variable $\frac{z}{\sqrt{\epsilon}}$ is still denoted by $z$ for simplicity of notation. And the pressure and velocity  of the
outer flow denoted by  $p^E(t,x,y)$ and ${\mathbf u}^E=(u^E,v^E,0)(t,x,y)$  satisfy the
Bernoulli's law
$$
\partial_t {\mathbf u}^E+({\mathbf u}^E\cdot \nabla){\mathbf u}^E +\nabla p^E=0.
$$

For later presentation, we denote  the Prandtl operator by
$$
P^\mu=\partial_t+u\partial_x+v\partial_y +w\partial_z-\mu\partial_z^2,
$$
with a parameter $\mu$ in front of the dissipation in the normal direction. 

Note that from the no-slip boundary condition and the incompressibility, we have
$$
w=-\int_0^z (u_x+v_y)dz,
$$
so that the Prandtl equations can be written as
$$
\begin{cases} 
\big(\partial_t+u\partial_x+v\partial_y -\int_0^z (u_x+v_y)dz\partial_z-\mu\partial_z^2\big)u=-\partial_xp^E(t,x,y,0),\\
\big(\partial_t+u\partial_x+v\partial_y -\int_0^z (u_x+v _y)dz\partial_z-\mu\partial_z^2\big) v=-\partial_yp^E(t,x,y,0).
\end{cases}
$$
From the above two time evolution equations on $(u,v)$, the loss of tangential 
derivative is obvious because of the
	convection term on the left hand side has a non-local term that contains tangential
	derivatives of $(u,v)$. In fact, whether there is a general well-posedness theory 
	in three space dimensions
	with finite order differential regularity  remains unsolved in contrast to  the classical work by Oleinik
	in 1963 in two space dimensions under the monotonicity condition, cf. \cite{oleinik-3} and the
	references therein. Precisely, under the monotonicity condition, the Crocco transformation is used in the classical work
	by Oleinik in which the normal coordinate $z$ is replaced by $u$ so that the reduced equation
	for $u^2$ becomes degenerate parabolic. And then the maximum principle argument can be
	applied with subtle analysis. 
	
	On the other hand, with infinite order of differential regularity, the well-posedness of the Prandtl equations
	was proved in the seminal work by Sammartino-Caflisch \cite{Samm} in analytic framework,  and in recent work  \cite{DG, LMY} in Gevrey function space with optimal
	index 2   in two and three space dimensions.
	
The well-posedness in analytic framework can be illustated as follows.
Consider a time evolution equation
$$
P^\mu u=f.
$$
Let $\partial^m u$ be the m-th order tangential derivatives of $u$ so that it satisfies
$$
P^\mu\partial^m u =F(\partial^{m+1}u, \cdots),
$$
with the source term depending on $\partial^{m+1}u$ up to one order power.
In order to obtain an estimate on the analytic norm on $u$, we only need to have a local in time
bound on $\frac{\rho^m\|\partial^m u\|}{ m!}$, where $\rho=\rho(t)$ is the radius of analyticity.
Based on a basic inequality
$$
m (\frac{\tilde{\rho}}{\rho})^m\lesssim \frac{1}{\rho-\tilde{\rho}},\quad \tilde{\rho}<\rho,$$
one can estimate the source term with one extra order of derivative by
\begin{eqnarray*}
\int_0^{t} \frac{\rho^m(t)\|\partial^{m+1}u\|(s)}{m!}ds
\lesssim \|u\|\int_0^{t} (\frac{\rho(t)}{\rho(s)})^m(m+1) ds\lesssim \int_0^{t} \frac{\|u\|}{\rho(s)-\rho(t)}ds.
\end{eqnarray*}
Then by choosing a suitable radius function of analyticity in time to make the final integral
bounded in finite time, the a priori bound can be closed by    an argument
using the  abstract Cauchy-Kowalewski theory. Please refer to \cite{Asano,Samm} for details.

With the above understanding on the Prandtl operator, we will investigate the  cancellation 
mechanisms of this operator through  directional derivatives in the next section.

\section{Cancellation mechanisms}
In this section, we will first recall the main observations in the papers \cite{awxy,MW} about the cancellations by using either  the convection term or the vorticity 
equation in two space dimensions.  And then we will present a new observation 
about directional derivatives
through some suitably chosen  vector fields of cancellation. 
The vector fields of cancellation are shown to be consistent with the recent
work on both the classical Prandtl operator and  the Prandtl operator derived from the MHD system
in the fully nonlinear regime.

Recall the Prandtl operator in two space dimensions $(x,z)$ given by
$$
\partial_t u+  u\partial_x u+w\partial_z u + \partial_x P^E=\mu\partial^2_z u.
$$
Consider its linearization around a divergence free vector field $(\tilde{u}, \tilde{w})$
$$
\partial_t u+  \tilde{u}\partial_x u+\tilde{w}\partial_z u +u\partial_x\tilde{u}
	+w\partial_z\tilde{u}=\mu\partial^2_z u +S,
$$
where $S$ represents the source term.
To treat the loss of derivative term of the unknown function $u$, that is $w\partial_z\tilde{u}$,
we can divide both sides by $\partial_z\tilde{u}$ 
under Oleinik's monotonicity condition on the velocity field $(\tilde{u},\tilde{w})$, i.e.,  $\tilde{\omega}=\partial_z \tilde{u}\neq 0$.  When we consider the time evolution of $u_x$  by differentiating the above equation in $x$,  the differentiation in $z$ again yields a cancellation by the divergence free condition on the velocity field $(u,w)$ for the two terms involving tangetial derivative of the second
order:
$$
(\frac{\tilde{u}\partial_{xx} u+w_x\partial_z\tilde{u}}{\partial_z\tilde{u}})_z=u_{xx}+w_{xz}+\tilde{R}=\tilde{R},
$$
where $\tilde{R}$ contains terms with  tangential derivative of $u$ at most one order. This implies that
one can use the good unknown function
$$
g_1=(\frac{u_x}{\tilde{\omega}})_z=\frac{w_x\tilde{w}-\tilde{w}_z u_x}{\tilde{w}^2},
$$	
with $w=\partial_z u$ to avoid the loss of tangential derivative in the time evoluation equation. And this idea is used in \cite{awxy} through the  Nash-Moser iteration to yield the local
in time well-posedness.

Another cancellation function observed in \cite{MW} is by noticing the vorticity equation for $\omega$ has the same form as $u$, that is
$$
P^\mu(\omega)=0,\quad P^\mu(u)=-\partial_x P^E.
$$
Hence 
$$
P^\mu(\omega_x)=-u_x\omega_x-w_x\omega_z,\quad P^\mu(u)=-u_x^2-w_x\omega-\partial^2_x P^E,
$$
where the non-local term $w_x$ containing extra one order of tangential derivative can be cancelled by
using the good unknown function
$$
f_1=\omega_x-\frac{\omega_z}{\omega}u_x,
$$	
 cf. \cite{MW} for details. Since $f_1\sim \omega g_1$, 
under the Oleinik monotonicity condition $\omega\neq 0$, the two good unknown functions
$g_1$ and $f_1$ are basically
similar up to a weight function.

We now introduce the concept of the field of cancellation so that the above cancellation becomes 
clear and more physical.

\begin{definition}
A vector field $\Theta$ is called a field of cancellation for  loss of tangential
 derivative with respect to the
Prandtl operator $P^\mu$   if the commutator of $P^\mu$ and $\Theta\cdot \nabla$  does not have 
the loss
of tangential derivative property. That is, for any differential function $f$, 
$$
[P^\mu, \Theta\cdot \nabla]f=R,
$$
where $R$ contains tangential derivative of $u$ and $f$ up to the first order.
\end{definition}
\begin{remark}
	Note that $(1,0)\cdot \nabla u$ can not be estimated directly in the Prandtl equation. However, if there
	exists a field of cancellation $\Theta$, then 
	$\Theta\cdot\nabla u$ can be estimated. Hence,  the remained question is whether one can recover $(1,0)\cdot \nabla u$ from $\Theta\cdot\nabla u$. For this, some structural assumption is needed.
\end{remark}

For the existence of field of cancellation for Prandtl operator, we have the following lemma.

\begin{lemma}\label{main}
	Assume $\mathbf{u}=(u,w)$ is a divergence free vector field in 2D and $P^\mu$ is  the 
	Prandtl operator.  If there exists  a vector function $\Theta$ that contains no tangential derivatives
	of $u$ satisfying
	$$
	P^\mu \Theta =(\Theta\cdot \nabla) \mathbf{u},
	$$
	then $\Theta$ is a field of cancellation.
	\end{lemma}
\begin{proof}
	Note that

	\begin{eqnarray*}
	&& P^\mu(\Theta \cdot \nabla f)= (P^\mu \Theta) \cdot \nabla f+ \Theta \cdot [P^\mu,\nabla] f + 
	(\Theta \cdot \nabla) P^\mu f\\[3mm]
	&&= (\Theta\cdot\nabla )\mathbf{u}\cdot \nabla f -\Theta \cdot \begin{pmatrix}
		u_x\partial_x +w_x\partial_z \\[2mm]
		u_z\partial_x +w_z\partial_z
	\end{pmatrix}f + (\Theta \cdot \nabla) P^\mu f\\[3mm]
&&= \theta_1 w_x f_z-\theta_1 w_x f_z + (\Theta \cdot \nabla) P^\mu f+R\\
&&=(\Theta \cdot \nabla) P^\mu f+R,
	\end{eqnarray*}
where $\Theta=(\theta_1,\theta_2)$, that is
$$
[P^\mu, \Theta\cdot \nabla]f=R,
$$
where $R$ contains  tangential derivatives of $u$ and $f$ up to the first order. Here, note that $w_x$ that is related to the second order derivative of $u$ in $x$ is cancelled.
	\end{proof}

The above lemma provides a strategy of finding a  vector field for recovering the loss of tangential 
derivatives.
That is,  if we estimate the directional derivative $\Theta\cdot\nabla f$ using
the Prandtl operator, there is no loss of tangential derivative. On the other hand, it is crucial that one can
recover the tangential derivative $(1,0)\cdot\nabla f$. For this, one needs some structural assumption
such as the Oleinik's monotonicity condition. In addition, for higher order tangential derivative, we can
use the terms in the $\partial_x^m (\Theta\cdot\nabla f)$ that involve highest order tangential
derivatives.

In the following two subsections, we will present the existence of the
vector field $\Theta$ for two physical models, that is, the classical Prandtl operator and the Prandtl operator derived from the MHD system in two space dimensions. Note that it is a very interesting
and unsolved problem about whether such vector field exists in three space dimension.

\subsection{2D Prandtl equation} If we consider the classical Prandtl equation, a  field of
cancellation can be
constructed as follows. First of all, 
by $P^\mu(u_z)=0$, we have
\begin{eqnarray*}
P^\mu(u_{zz})&=& -u_zu_{zx}-w_zu_{zz}\\
&=&u_xu_{zz}+w_{zz}u_z=\mathbf{u}_{zz}\cdot \nabla u.
\end{eqnarray*}
On the other hand
$$
P^\mu(w_z)=-P^\mu(u_x)=P^E_{xx}+u^2_x+w_xu_z,
$$
gives
\begin{eqnarray*}
	P^\mu(w_{zz})&=& (u^2_x+w_xu_z)_z-u_zw_{xz}-w_zw_{zz}\\
	&=&w_xu_{zz}+w_{zz}w_z=\mathbf{u}_{zz}\cdot \nabla w.
\end{eqnarray*}
Hence,
$$P^\mu(\mathbf{u}_{zz})=(\mathbf{u}_{zz}\cdot \nabla )\mathbf{u}.$$
According to the Lemma 2.3, we can set
$$\Theta=\mathbf{u}_{zz},$$
so that 
$$
P^\mu(\mathbf{u}_{zz}\cdot\nabla \mathbf{u})=R,
$$
where the source term $R$ contains   tangential derivative of $u$ at most one order. Therefore, standard analytic techniques can be applied to the above equation for desired estimates on
$\mathbf{u}_{zz}\cdot\nabla u$. Note that $(1,0)\cdot\nabla u=u_x$ can be recovered from $\mathbf{u}_{zz}\cdot\nabla u$  if $u_z\neq 0$ because 
$$
\mathbf{u}_{zz} \cdot \nabla u=u_{zz}u_x-u_{xz}u_z=-\omega f_1\sim -\omega^2 g_1,
$$
where $g_1$ and $f_1$ are the two good unknown functions used in \cite{awxy,MW} mentioned above.
This shows that under the 
Oleinik's monotonicity condition $\omega\neq 0$,  the directional derivative $\mathbf{u}_{zz}\cdot\nabla u$ can be used to recover  the tangential derivative $(1,0)\cdot\nabla u=u_x$ as in  \cite{awxy,MW}.

\subsection{2D MHD}

In this subsection, we will present another model to illustrate the
existence of  vector field of cancellation. For this, consider the MHD model in two space dimensions:
\begin{align}
\label{1.1}
\left\{
\begin{array}{ll}
\partial_t \mathbf{u}+\mathbf{u}\cdot\nabla \mathbf{u}+\nabla p-\frac{1}{\hbox{Re}}\triangle \mathbf{u}=S \mathbf{h}\cdot\nabla \mathbf{h},\\
\partial_t \mathbf{h}-\hbox{curl}(\mathbf{u}\times \mathbf{h})+\frac{1}{\hbox{Rm}}\hbox{curl curl} \mathbf{h}=\mathbf{0},\\
\hbox{div} \mathbf{u}=0,\quad \hbox{div} \mathbf{h}=0,\quad (x,y)\in\Omega=\mathbb{T}\times\mathbb{R}_+,
\end{array}
\right.\nonumber
\end{align}
where $\mathbf{u}=(u,w)$ and $\mathbf{h}=(f,h)$ represent the velocity and magnetic fields respectively, and $p$ is the total pressure. Here, there are some physical parameters, 
$Re$ representing the Reynolds number, $Rm$ the magnetic Reynolds number, and 
$S=\frac{\mbox{Ha}^2}{\mbox{Re}\mbox{Rm}}$ the  coupling parameter, and
$\mbox{Ha}$ the Hartmann number.

It is known that in the nonlinear regime when $\mbox{Re}\sim \mbox{Rm}\sim \mbox{Ha}$ being
sufficiently large, one can derive
a Prandtl type boundary layer system of equations with no-slip boundary condition on the velocity field and
perfect conducting condition on the magnetic field. Precisely, by taking
$$\mbox{Rm}=\frac{1}{\kappa\epsilon}, \,\, \mbox{Re}=\frac{1}{\mu\epsilon},\,\,S=1,
$$
 with 
 $\epsilon$ being a small parameter, the MHD system becomes
\begin{equation*}
\label{MHD}
\begin{cases}
\partial_t\mathbf{u}^{\epsilon}+(\mathbf{u}^{\epsilon}\cdot\nabla)\mathbf{u}^{\epsilon}-(\mathbf{h}^{\epsilon}\cdot\nabla)\mathbf{h}^{\epsilon}+\nabla p^\epsilon=\mu\epsilon\triangle\mathbf{u}^{\epsilon},\quad (x,y)\in\Omega,\\
\partial_t\mathbf{h}^{\epsilon}+(\mathbf{u}^{\epsilon}\cdot\nabla)\mathbf{h}^{\epsilon}-(\mathbf{h}^{\epsilon}\cdot\nabla)\mathbf{u}^{\epsilon}
=\kappa\epsilon\triangle \mathbf{h}^{\epsilon},\\
\nabla\cdot\mathbf{u}^{\epsilon}=0,\quad\nabla\cdot \mathbf{h}^{\epsilon}=0,\\
\mathbf{u}^{\epsilon}|_{y=0}=\mathbf 0,\qquad \partial_y f^\epsilon|_{y=0}=0,\quad h^\epsilon|_{y=0}=0,\\
(\mathbf{u}^{\epsilon},\mathbf{h}^{\epsilon})|_{t=0}=(\textbf{u}_0,\textbf{h}_0)(x,y).
\end{cases}
\end{equation*}
If one applies the Prandtl ansatz to the above system
\begin{align*}
\left\{
\begin{array}{ll}
u^\epsilon(t,x,z)=u(t,x, \frac{z}{\sqrt{\epsilon}}),\\
w^\epsilon(t,x,z)=\epsilon^{\frac{1}{2}}w(t,x,\frac{z}{\sqrt{\epsilon}} ),
\end{array}
\right.
\qquad
\left\{
\begin{array}{ll}
f^\epsilon(t,x,z)=f(t,x, \frac{z}{\sqrt{\epsilon}}),\\
h^\epsilon(t,x,z)=\epsilon^{\frac{1}{2}}h(t,x, \frac{z}{\sqrt{\epsilon}}),
\end{array}
\right.
\end{align*}
and
\[p^\epsilon(t,x,z)=p(t,x,\frac{z}{\sqrt{\epsilon}}),\]
the following Prandtl equations for MHD can
be derived:
	\begin{align*}
	\left\{
	\begin{array}{ll}
	\partial_tu+u\partial_xu+w\partial_z u-\mu\partial^2_z u=f\partial_xf+h\partial_zf-P_x,\\
	\partial_t f+u\partial_x f+w\partial_z f- \kappa\partial_z^2 f=f\partial_x u+h\partial_z u,\\
	\partial_xu+\partial_z w=0,\quad \partial_x f+\partial_z h=0,\\
	u|_{t=0}=u_{0}(x,y),\quad f|_{t=0}=f_{0}(x,y),\\
	(u,w,\partial_z f,h)|_{z=0}=0,\\
	\lim\limits_{z\rightarrow+\infty}(u,f)=(u^E, f^E)(t,x,0),
	\end{array}
	\right.
	\end{align*}
	where again the fast variable $\frac{z}{\sqrt{\epsilon}}$ is still denoted by $z$, the outer flow $(u^E, f^E, P)(t,x,0)$ is the trace of a solution to  the ideal MHD system on the boundary. For this system, note that the stream function $\psi$ of the magnetic field
	$(f,h)$ satisfies
	$$
	P^\kappa \psi =0,
	$$
	that is in analogue to the vorticity $\omega$ for the 2D Prandtl equation.
	The following two good unknown functions are used in \cite{LXY1} to take care of the
	$m$-th tangential derivatives of $u$ and $f$:
	\begin{equation}\label{umfm}
	u^m:=\partial_x^m u-\frac{\partial_zu}{f}\partial_x^m\psi,~ f^m:=\partial_x^m f-\frac{\partial_z f}{f}\partial_x^m\psi.
	\end{equation}
	With these unknown functions, the equations for $u^m$ and $f^m$ are in  the following
	 symmetric form
	so that the loss of derivatives can be treated
	\begin{align*}
	\left\{
	\begin{array}{ll}
	&\partial_t u^m +(u\partial_x +w\partial_z) u^m
	-(f\partial_x +h\partial_z) f^m= \mu\partial^2_z u^m +R_1\\[5mm]
	& \partial_t f^m +(u\partial_x +w\partial_z) f^m
	-(f\partial_x +h\partial_z) u^m=\kappa \partial^2_z f^m +R_2,
	\end{array}
	\right.
	\end{align*}
	where $R_i$, $i=1,2$, contain tangential derivatives of at most $m$-th order. Here,
	the non-degeneracy of the tangential magnetic field $f\neq 0$ is needed for both
	the definition of the unknown functions and also for recovering the tangential derivatives
	of $u$ and $f$ from them. Note that the coordinate transformation $(x,z,t)\rightarrow (x,\psi, t)$
	can play a role as the  Crocco transformation so that the reduced system is quasilinear and symmetric 
	and the standard analysis can then be applied.
	
	Let us  follow the Lemma \ref{main} to find out whether there is appropriate
	vector  field for cancellation to avoid the lost of tangential derivative difficulty. In fact, for the Prandtl system of MHD, this vector field is already built in as it is the direction of the magnetic field. In fact,
	note that
	$$
	P^\kappa (\mathbf{h}) =(\mathbf{h}\cdot \nabla )\mathbf{u}.
	$$
	The $\Theta$ in  Lemma \ref{main} is simply $\mathbf{h}$ that is intrinsic  in the system. Since
the Prandtl system for MHD is
	$$
	\begin{pmatrix}
	P^\mu u\\
	P^\kappa f
	\end{pmatrix}=\begin{pmatrix} 0&\mathbf{h}\cdot\nabla\\
	\mathbf{h}\cdot \nabla& 0\end{pmatrix} \begin{pmatrix} u\\f\end{pmatrix}.$$
	By applying the Lemma, we have 
	$$
	\begin{pmatrix}
	P^\mu (\mathbf{h}\cdot\nabla u)\\
	P^\kappa (\mathbf{h}\cdot \nabla f)
	\end{pmatrix}=\begin{pmatrix} 0&\mathbf{h}\cdot\nabla\\
	\mathbf{h}\cdot \nabla& 0\end{pmatrix} \begin{pmatrix} \mathbf{h}\cdot\nabla u\\
	\mathbf{h}\cdot\nabla f\end{pmatrix}+R,$$
	where $R$ contains tangential derivatives of $u$ and $f$ up to the first oder. Hence, 
	  $\mathbf{h}\cdot \nabla$ gives the direction  of the cancellation that is the direction of the
	  magnetic field. And the structural assumption of the non-degenerate tangential magnetic
	  field component, $f\neq 0$ is used to recover the tangential derivative $\begin{pmatrix} 1&0\end{pmatrix}\cdot (\nabla u, \nabla f)$. Note that
	$$
	(\mathbf{h}\cdot\nabla u, \mathbf{h}\cdot\nabla f)=f(u^1,f^1),
	$$
	where $(u^1,f^1)$ is the good function of the first order defined in \eqref{umfm} so that it is consistent with the observation in \cite{LXY1}.
	
	\vspace{1cm}

\noindent{\bf Acknowledgment:}  The research was supported by the General Research Fund of Hong Kong CityU No. 11303521.

\end{document}